%% file: nuclei_complex.tex
\documentclass[11pt]{amsart}

\usepackage{amsmath, amsthm, amssymb,mathtools}

\usepackage{palatino} 
\usepackage{setspace}

\usepackage[abs]{overpic}                 

\usepackage[mathscr]{eucal}

\usepackage{tikz}
\usepackage{tikz-cd}
\usepackage{calc}
\usepackage{float}
\usetikzlibrary{calc}
\tikzstyle{vertex}=[circle, draw, fill=black, inner sep=0pt, minimum size=6pt]
\tikzstyle{redvert}=[circle, draw, gray, fill=gray,  inner sep=0pt, minimum size=6pt]

\usepackage{xcolor}
\definecolor{indigo}{rgb}{0.29, 0.0, 0.51}
\usepackage[colorlinks, urlcolor=indigo, linkcolor=indigo, citecolor=indigo]{hyperref}

\makeatletter
\newcommand*\bigcdot{\mathpalette\bigcdot@{0.6}}
\newcommand*\bigcdot@[2]{\mathbin{\vcenter{\hbox{\scalebox{#2}{$\m@th#1\bullet$}}}}}
\makeatother

\newtheorem{theorem}{Theorem}

\newtheorem{obs}[theorem]{Observation}
\newtheorem{conjecture}[theorem]{Conjecture}

\theoremstyle{definition}
\newtheorem{definition}[theorem]{Definition}

\theoremstyle{remark}

\numberwithin{theorem}{section}
\theoremstyle{plain}

\newcommand{\vertex}{\node[vertex]}
\newcommand{\redvert}{\node[redvert]}

\begin{document}

\title{On Elser's conjecture and the topology of $U$-nucleus complex}

\author{Apratim Chakraborty}

\author{Anupam Mondal}

\author{Sajal Mukherjee}

\author{Kuldeep Saha}

\email{apratim.chakraborty@tcgcrest.org}

\email{anupam.mondal@tcgcrest.org}

\email{sajal.mukherjee@tcgcrest.org}

\email{kuldeep.saha@tcgcrest.org}

\address{TCG CREST, Sector V, Bidhannagar, Kolkata, West Bengal, India 700091}

\maketitle

\begin{abstract}	

Dorpalen-Barry et al. proved Elser's conjecture about signs of Elser's numbers by interpreting them as certain sums of reduced Euler characteristic of an abstract simplicial complex known as  $U$-nucleus complex. We prove a conjecture posed by them regarding the homology of $U$-nucleus complex.  
 
\end{abstract}

\section{Introduction}

Given a connected undirected simple graph $G=(V(G),E(G))$ with at least $3$ vertices, a nucleus of $G$ is a connected subgraph $N \subseteq G$ such that $V(N)$ is a vertex cover of $G$. The set of nuclei is denoted by $\mathcal{N}(G)$ \cite{nucleus}. The $k$-th Elser number $\operatorname{els}_k(G)$ is defined as follows

\begin{align*}
	\operatorname{els}_k(G)=(-1)^{|V(G)|+1} \sum_{N \in \mathcal{N}(G)} (-1)^{|E(N)|} |{V(N)}|^{k}.
\end{align*}

Elser's conjecture \cite{Elser} says $\operatorname{els}_k(G) \ge 0 $ for $k \ge 2$, $\operatorname{els}_1(G) = 0 $ and $\operatorname{els}_0(G) \le 0 $.\\ 

An abstract simplicial complex is a (finite, nonempty) family of finite sets that is closed under taking subsets. Note that the empty set always belongs to the family. Given an abstract simplicial complex $\Sigma$, the \emph{reduced Euler characteristic}, $\tilde{\chi}( \Sigma):= \sum_{\sigma \in \Sigma} {(-1)}^{|\sigma|-1}$. So, it follows that $\tilde{\chi}( \Sigma) = \chi( \Sigma) - 1$, where  $\chi( \Sigma)$ denotes the usual Euler characteristic of $\Sigma$.\\ 

This conjecture was proved in \cite{nucleus}. That proof made use of an abstract simplicial complex called the $U$-nucleus complex $\Delta^{G}_{U}$ as defined below.

\[ 
\text{For } U \subseteq V(G), \Delta^{G}_{U}= \{ E(G) \setminus E(N): N \in \mathcal{N}(G), \ U \subseteq V(N) \}. 
\]	

The $k$-th Elser number, $\operatorname{els}_k(G)$ can be seen as a weighted sum of reduced Euler characteristic of $\Delta^{G}_{U}$ as below.

\begin{align*}
	\operatorname{els}_k(G)=(-1)^{|E(G)|+|V(G)|} \sum_{U \subseteq V(G)} \operatorname{Sur}(k,|U|) \ \tilde{\chi}(\Delta^{G}_{U})
\end{align*}

where, $\operatorname{Sur}(a,b)$ denotes the number of surjections from a set of size $a$ to a set of size $b$. Dorpalen-Barry et al. proved Elser's conjecture by establishing analogous inequalities for the reduced Euler characteristic $\tilde{\chi}(\Delta^{G}_{U})$.  They posed a general conjecture (Conjecture 9.1 in \cite{nucleus}) about the reduced homology of these complexes that implies the results about the reduced Euler characteristic as a corollary. We recall that the $k$-th reduced homology of a complex $\Sigma$, denoted by $\tilde{H}_k(\Sigma;\mathbb R)$, is same as the usual simplicial homology of $\Sigma$ for all $k \ge 1$ and the zeroth simplicial homology of $\Sigma$ is $\tilde{H}_0(\Sigma;\mathbb R) \oplus \mathbb{R}$. Now, we state the conjecture mentioned above.

\begin{conjecture} \cite{nucleus}\label{conj}
Let $G$ be a connected graph and $U \subseteq V(G)$. Then the reduced homology
group $\tilde{H}_k(\Delta^{G}_{U}; \mathbb{R})$ is nonzero only if (i) $U=\emptyset$ and $k =|E(G)|-|V (G)|+1$, or (ii) $|U|>1$ and $k=|E(G)|-|V (G)|$. \footnote{ We note that part (i) of Conjecture~\ref{conj} originally stated  that $\tilde{H}_k(\Delta^{G}_{U}; \mathbb{R})$ is nonzero for $U= \emptyset$ only if $k =|E(G)|-|V (G)|-1$ (Conjecture~9.1 in \cite{nucleus}). However, the $n$-cycle graph provides a counterexample in that case. Through personal communication, the authors have informed us that there was an error and they intended $k$ to be $|E(G)|-|V (G)|+1$. So, we have stated the corrected version of the conjecture as communicated by the authors.}
	
\end{conjecture}

We prove the following theorem about homology of $\Delta^{G}_{U}$ settling the conjecture in the case of $|{U}|>1$. 

\begin{theorem} \label{maintheorem}
For $|U|>1$, $\tilde{H}_{k}(\Delta^{G}_{U};\mathbb{R}) \cong 0$ unless $k=|E(G)|-|V(G)|$. 

\end{theorem}

Our approach to prove Theorem~\ref{maintheorem} is inspired by Grinberg's proof of Elser's conjecture for the case $k=1$ (i.e., $\operatorname{els}_1(G)=0$) using discrete Morse theory \cite{Grinberg}. For $F \subseteq E(G)$, an $F$-path is a path consisting of edges from $F$. Now given a vertex $v\in V(G)$, we define $\operatorname{Shade}_v (F)$ as follows. 

\[ \operatorname{Shade}_v (F)=\{ e \in E(G) : \text{ There is an } F \text{-path from an endpoint of } e \text { to } v \}. \]

We consider the following abstract simplicial complex as defined in \cite{Grinberg} 

\[ \mathcal{A}_U = \{ F \subseteq E(G)  :  \operatorname{Shade}_v(F) \subsetneq E(G) \text{ for some } v \in U \}. \]

It follows that $\mathcal{A}_U $ is the Alexander dual of $\Delta^{G}_{U}$ \cite{Grinberg}. Grinberg showed that $\mathcal{A}_U $  is collapsible when $|U|=1$ by producing an acyclic matching (also known as gradient vector field) with no unmatched simplices. By Alexander duality, it follows that in the case of $|U|=1$, for all $i$, $\tilde{H}_i(\Delta^{G}_{U}) \cong \tilde{H}^{|E(G)|-i-3}(\mathcal{A}_U) \cong \tilde{H}_{|E(G)|-i-3}(\mathcal{A}_U) \cong 0  $. Therefore, $\Delta^{G}_{U}$ has homology of a point. \\

We observe that $\mathcal{A}_U = \cup_{x \in U} \mathcal{A}_x$.  We extend Grinberg's matching to $\mathcal{A}_U $ when $|U|>1$ and show that $\mathcal{A}_U $ can be given an acyclic matching  with unmatched simplices in the dimension $k=|V(G)|-3$. Again by Alexander's duality we conclude that, $\Delta^{G}_{U}$ has homology concentrated in a single homological degree $|E(G)|-|V(G)|$.\\

\textbf{Acknowledgments} The authors would like to thank Goutam Mukherjee for helpful discussions.

\section{Preliminaries}

\subsection{Graph theoretic notation}

\begin{itemize}

\item A (simple undirected) graph $G$ is an ordered pair $(V(G), E(G))$, where $V(G)$ is a finite set and $E(G) \subseteq \{ S \subseteq V(G): \ |S|=2 \}$. Elements of $V(G)$ are called vertices and elements of $E(G)$ are called edges. \\

\item We call a vertex $y$ a \emph{neighbor} of $x$ in $G$ if $x$ and $y$ are adjacent, i.e.,  $\{ x,y \} \in E(G)$. \\

\item A \emph{leaf} in $G$ is a vertex with only one neighbor.\\

\item A \emph{leaf-edge} denotes an edge which is incident to (i.e., contains) a leaf.\\

\item For a connected graph $G$, an edge $e$ is called a \emph{bridge} if the graph obtained after deleting $e$ is disconnected.\\

\item  \emph{Edge-induced subgraph}: Let  $\sigma \subseteq E(G)$. We define the subgraph induced by $\sigma$, denoted by $G_\sigma$ as follows
 \[ V(G_{\sigma})= \{ x \in V(G) : x \in e \text{ for some } e \in \sigma \} \text{ and }
E(G_{\sigma})= \sigma.\]

We use $V_{\sigma}$ to denote  $V(G_{\sigma})$ for brevity.\\

\item \emph{Vertex cover (of edges)} of $G$: $X \subseteq V(G) \text{ such that for all } e \in E(G), e\cap X \neq \emptyset$.


\end{itemize}

\subsection{Basics of discrete Morse theory}

In the case of smooth Morse theory, existence of a Morse function is equivalent to the existence of a gradient vector field. Similarly, existence of a \emph{discrete Morse function} on a simplicial complex is equivalent to the existence of a \emph{discrete gradient vector field}. First we recall the notion of a \emph{discrete vector field} or \emph{matching} on a simplicial complex. Let $K$ be a simplicial complex. 

\begin{definition}[Discrete vector field / matching]
	
	A discrete vector field $V$ on $K$ is a collection of pairs $\{\alpha^{(p)} < \beta^{(p+1)}\}$ of simplices of $K$ such that each simplex is in at most one pair of $V$.
	
\end{definition}   

Pictorially, given a discrete vector field on $K$, we assign arrows on $K$ such that for a pair $\{\alpha^{(p)} < \beta^{(p+1)}\}$ the head of the arrow lies in $\beta^{(p+1)}$ and the tail of the arrow lies in $\alpha^{(p)}$.  A gradient vector field is a discrete vector field with some special properties about this arrows.

\begin{definition}[Gradient vector field / acyclic matching]
	A discrete vector field is called a gradient vector field if given any simplex $\alpha$ in $K$, it satisfies exactly one the following.

	\begin{enumerate}
		\item $\alpha$ is the tail of exactly one arrow.
		
		\item $\alpha$ is the head of exactly one arrow.
		
		\item $\alpha$ is neither the head nor the tail of an arrow.
	\end{enumerate}

\end{definition}

\begin{figure}[htbp] 
	\centering
	\def\svgwidth{8cm}
	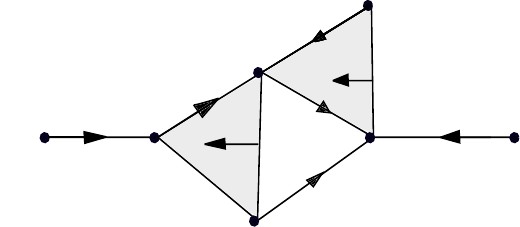
	\caption{An example of a discrete gradient vector field.}\label{gradfield}
\end{figure}

According to this notion of discrete Morse functions in terms of discrete gradient fields, a simplex is \emph{critical} if and only if it is neither the tail nor the head of any arrow. For instance, in Figure \ref{gradfield}, $e_2$ is a critical $1$-simplex while $e_1$ is not critical. The criteria when a discrete vector field is a gradient vector field is more straight forward once we have the notion of a \emph{$V$-path}.

\begin{definition}
	Given a discrete vector field $V$ on a simplicial complex $K$, a $V$-path is a sequence of simplices $$\alpha_0^{(p)}, \beta_0^{(p+1)}, \alpha_1^{(p)}, \beta_1^{(p+1)},\cdots,\beta_r^{(p+1)}, \alpha_{r+1}^{(p)}$$ such that for each $i \in \{0,1, \cdots, r\}$, $\{\alpha_i < \beta_i \} \in V$ and $\beta_i > \alpha_{i+1} \neq \alpha_i$.

\end{definition}

\noindent 

We say a path is a non-trivial \emph{closed path} if $r \geq 0$ and $\alpha_0 = \alpha_{r+1}$. It can be shown that a gradient vector field does not have a non-trivial closed $V$-path (i.e., a cycle). Moreover, the other direction is also true.

\begin{theorem}[Forman \cite{Forman}] \label{v-path criteria}
	A discrete vector field $V$ is the gradient vector field of a discrete Morse function if and only if there are no non-trivial closed $V$-path.
\end{theorem} 

In other words we say a discrete vector field $V$ is gradient vector field if the corresponding matching is acyclic (i.e., it contains no closed $V$-path).\\
 
\noindent The notion of gradient vector field is closely related to the notion of \emph{simplicial collapse}. A simplex $\alpha$ is called a face of a simplex $\beta$ if $\alpha \subsetneq \beta$. A simplex which is a face of exactly one simplex is called a \emph{free face}. For example, in Figure \ref{gradfield}, $v_1$, $e_1$ and $e_2$ are free faces while $v_2$ is not. Whenever we have a free face $\alpha$ of a simplex $\beta$ in a simplicial complex $K$, we remove $\alpha$ and $\beta$ (but keeping faces of $\beta$ other than $\alpha$ unperturbed) from $K$  by a deformation retraction. This is known as \emph{elementary collapse}. If $K \setminus \{ \alpha, \beta\}$ is obtained from $K$ by a elementary collapse, we extend any gradient vector field of $V$ of $K \setminus \{ \alpha, \beta\}$ to $K$ by adding the pair $(\alpha,\beta) $ to $V$.

\begin{figure}
	\def\svgwidth{12cm}
	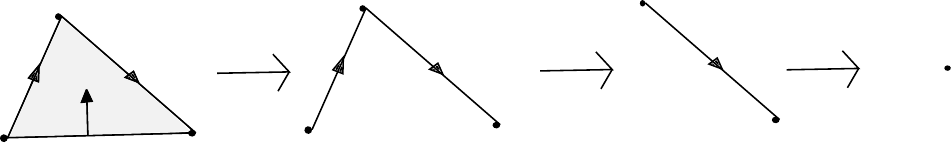
	\caption{A sequence of elementary collapses.} \label{collapse}
\end{figure}

If one can go from a simplicial complex $K_1$ to another simplicial complex $K_2$ via a sequence of elementary collapses, then we say $K_1$ collapses to $K_2$. If a gradient vector field on a simplicial complex $K$ does not have any critical faces of dimension greater than $0$, then $K$ collapses to a point.  We also say $K$ is collapsible. \\

Now, we state the fundamental theorem of discrete Morse theory.

\begin{theorem}[Forman \cite{Forman}]  \label{forth}
Suppose $K$ is a simplicial complex with a discrete Morse function. Then $K$ is homotopy equivalent to a CW complex with exactly one cell of dimension $p$ for each critical simplex of dimension $p$.
\end{theorem}

We also need the following crucial fact from discrete Morse theory which follows from the theorem above.

\begin{theorem}[Sphere theorem, \cite{J}] If $K$ has only critical faces of dimension $d$ ($\geq 1$), then $K$ is homotopy equivalent to a wedge of $d$-spheres. 
	
\end{theorem}

\subsection{Collapsibility of $\mathcal{A}_x$}

We consider the case $U= \{x \}$. Grinberg produced an acyclic matching on $\mathcal{A}_x$ with no critical simplices (here we adopt the convention that if only one $0$-simplex is unpaired in some matching, we are allowed to pair it with $\emptyset$). Then, we conclude by Theorem~\ref{forth} that $\mathcal{A}_x$ is collapsible.\\

We briefly explain the matching $\Phi_0$ constructed in Grinberg's proof. We fix an arbitrary ordering in $E(G)$. For any set $F \in \mathcal{A}_x $ define, $\sigma(F)$ to be the minimum edge in $E(G) \setminus \operatorname{Shade}_x(F)$ (this set is non-empty by definition). If $ \sigma(F) \notin F$, we pair off $F$ with $F \cup \sigma(F)$. This gives an acyclic matching in $\mathcal{A}_x $ with no critical simplices \cite{Grinberg}. Thus we have the following theorem.  

\begin{theorem} \cite{Grinberg}
For any $x\in V(G)$, $\mathcal{A}_x$ is collapsible.
\end{theorem}

\section{Proof of the main theorem}

\begin{proof}[Proof of Theorem \ref{maintheorem}]

Let $E(G)=\{ e_1, e_2,\cdots, e_m \}$. We fix the ordering $ e_1 < e_2 < \cdots < e_m $ on $E(G)$. The proof is by induction. For the base case, let $U=\{x,y\}.$ Clearly, $\mathcal{A}_U = \mathcal{A}_x \cup \mathcal{A}_y.$ We consider the acyclic matching  $\Phi_0$ on $\mathcal{A}_x$ with no critical simplices as mentioned before. Our aim is to extend this matching $\Phi_0$ to the whole simplicial complex $\mathcal{A}_U.$ In order to achieve our goal we proceed recursively. We divide this recursive process into several steps. Note that we want to define an acyclic matching on the simplices of $\mathcal{A}_y \setminus\mathcal{A}_x.$ First, we make the following observation to characterize the simplices in $\mathcal{A}_y \setminus\mathcal{A}_x$. 

\begin{obs} 
 Let $\sigma \subseteq E(G)$ and $\sigma \neq \emptyset$. Then, $ \sigma \in \mathcal{A}_y \setminus \mathcal{A}_x$ if and only if the following conditions hold 
 \begin{enumerate}
 \item $G_\sigma$ is a connected subgraph of $G$,
 \item $V_\sigma$ is a vertex cover of $G$,
 \item $x \in V_\sigma$ and  $y \notin V_\sigma$.
 
 \end{enumerate}
Also, $\emptyset \in \mathcal{A}_y \setminus\mathcal{A}_x$ if and only if $G$ is a star graph with $x$ as the central vertex.
\end{obs}

Now, we proceed to our initial step of the recursive process.

 \textbf{Step 1:} Consider those simplices $\sigma$ in $\mathcal{A}_y \setminus\mathcal{A}_x,$ such that, $e_1 \in \sigma$ and one of the following holds. 
 \begin{enumerate}
    \item $e_1$ is a part of a cycle (i.e., not a bridge) in $G_\sigma$ (see Fig. \ref{sigma1}).
    \item $e_1$ is a leaf-edge in $G_\sigma$ such that 
    \begin{enumerate}
        \item $e_1$ doesn't contain $x$ \emph{as a leaf} in $G_\sigma$ \\ and
        \item If  $z \in V(G) \setminus V_\sigma$, then $e_1$  doesn't contain any neighbor of $z$ \emph{as a leaf} in $G_\sigma$ (see Fig. \ref{sigma2}). 
\end{enumerate}         
\end{enumerate}

\begin{figure}[!ht]
\begin{minipage}{0.45\textwidth}
\[\begin{tikzpicture}[scale=1.5]
	\node at (-0.3,0) {$x$};
	\node at (2.3,0) {$y$};
	
	\node at (0.5,0.2) {$e_1$};
	\node at (1.2,-0.5) {$e_2$};
	\node at (0.2,-0.5) {$e_3$};
	
	\vertex (x) at (0,0) {};
	\vertex (w) at (1,0) {};
	\vertex (z) at (1,-1) {};
	
	\redvert (y) at (2,0) {};

	\path [line width=1.75pt]
		(x) edge (w)
		(x) edge (z)
		(w) edge (z)
	 ;  

	\path [gray,line width=1pt]
		(w) edge (y)
		(z) edge (y)
	;  

\end{tikzpicture}\]
\caption{ $\sigma = \{e_1,e_2,e_3\}$; $e_1$ is a part of a cycle in $G_\sigma$.} \label{sigma1}
\end{minipage}
\hfill
\begin{minipage}{0.45\textwidth}
	\[\begin{tikzpicture}[scale=1.5]
	\node at (-0.3,0) {$x$};
	\node at (1.3,1) {$y$};
	\node at (1.3,-1) {$z$};
	
	\node at (1.5,0.2) {$e_1$};
	\node at (0.5,0.2) {$e_2$};
	
	\vertex (x) at (0,0) {};
	\vertex (w1) at (1,0) {};
	\vertex (w2) at (2,0) {};
	
	\redvert (y) at (1,) {};
	\redvert (z) at (1,-1) {};
	
	\path [line width=1.75pt]
	(x) edge (w1)
	(w2) edge (w1)
	;  
	
	\path [gray,line width=1pt]
	(x) edge (y)
	(x) edge (z)
	(w1) edge (y)
	(w1) edge (z)
	;

	\end{tikzpicture}\]
	\caption{ $\sigma = \{e_1, e_2\}$; $e_1$ is a leaf-edge in $G_\sigma$ and $e_1$ doesn't contain any neighbor of $y$ or $z$ as a leaf.} \label{sigma2}
	\end{minipage}

\end{figure}

We note that in this case $\sigma \setminus \{ e_1 \} \in \mathcal{A}_y \setminus \mathcal{A}_x.$ We pair off $\sigma$ with $\sigma \setminus \{e_1\}.$ We mention an exceptional pairing in the special case of $ \{e_1 \} \in \mathcal{A}_y \setminus \mathcal{A}_x$ as well as $\emptyset \in \mathcal{A}_y \setminus \mathcal{A}_x$ (as mentioned before possible only if $G$ is a star graph with $x$ as the central vertex). In that case we pair off  $ \{e_1 \}$ with $\emptyset$.\\ 

We observe that conversely any $\sigma \in \mathcal{A}_y \setminus \mathcal{A}_x$, with $e_1 \notin \sigma$, is uniquely paired with $\sigma \cup \{ e_1 \}$ whenever $\sigma \cup \{ e_1 \} \in  \mathcal{A}_y \setminus \mathcal{A}_x$. Therefore, this extension of $\Phi_0$ is well defined and is denoted by $\Phi_1.$

\textbf{Step 2:} The set of all unpaired simplices after step $1$ is denoted by $\mathcal{C}_1.$ Note that $\mathcal{C}_1$ consists of simplices  $\tau$ in $\mathcal{A}_y \setminus\mathcal{A}_x,$ such that, either $e_1 \notin \tau$ (possible only if $y \in e_1$)  or $e_1 \in \tau$ and one of the following holds.

 \begin{enumerate}
    \item $e_1$  is a bridge but not a leaf-edge in $G_\tau$ (see Fig. \ref{sigma3}).
    \item $e_1$ is a leaf-edge in $G_\tau$ such that,
    \begin{enumerate}
         \item  $e_1$ contains $x$ as a leaf in $G_\tau$ (see Fig. \ref{sigma4}, \ref{sigma5})\\ or 
    \item  there exists $z \in V(G) \setminus V_\tau$ such that $e_1$  contains a neighbor of $z$ as a leaf in $G_\tau$ (see Fig. \ref{sigma5}, \ref{sigma6}). 

\end{enumerate}
\end{enumerate}

\begin{figure}[!ht]

\begin{minipage}{0.45\textwidth}
	\[\begin{tikzpicture}[scale=1.5]
	\node at (-0.3,0) {$x$};
	\node at (3.8,0) {$y$};
	
	\node at (0.5,0.2) {$e_2$};
	\node at (1.2,-0.5) {$e_5$};
	\node at (0.2,-0.5) {$e_4$};
	\node at (2,-0.5) {$e_1$};
	\node at (2.7,-0.5) {$e_3$};
	
	\vertex (x) at (0,0) {};
	\vertex (w1) at (1,0) {};
	\vertex (w2) at (2.5,-1) {};
	\vertex (w3) at (2.5,0) {};	
	\vertex (z) at (1,-1) {};
	
	\redvert (y) at (3.5,0) {};
	
	\path [line width=1.75pt]
	(x) edge (w)
	(x) edge (z)
	(w1) edge (z)
	(w1) edge (w2)
	(w3) edge (w2)
	;  
	
	\path [gray,line width=1pt]
	(w3) edge (y)
	(w2) edge (z)
	(w3) edge (w1)
	;

	\end{tikzpicture}\]
	\caption{ $\tau = \{e_1,\ldots,e_5\}$; $e_1$ is a bridge in $G_\tau$, but not a leaf-edge.}\label{sigma3}
\end{minipage}
\hfill
\begin{minipage}{0.45\textwidth}
		\[\begin{tikzpicture}[scale=1.5]
	\node at (-0.3,0) {$x$};
	\node at (2.3,0) {$y$};
	
	\node at (0.5,0.2) {$e_1$};
	\node at (1.2,-0.5) {$e_2$};
	
	\vertex (x) at (0,0) {};
	\vertex (w) at (1,0) {};
	\vertex (z) at (1,-1) {};
	
	\redvert (y) at (2,0) {};
	
	\path [line width=1.75pt]
	(x) edge (w)
	(w) edge (z)
	;  
	
	\path [gray,line width=1pt]
	(w) edge (y)
	(x) edge (z)
	;

	\end{tikzpicture}\]
	\caption{ $\tau = \{e_1,e_2\}$; $e_1$ is a leaf-edge in $G_\tau$ and $e_1$ contains $x$ as a leaf in $G_\tau$.}\label{sigma4}
	\end{minipage}
	
\end{figure}

\begin{figure}[!ht]
	
\begin{minipage}{0.45\textwidth}
	\[\begin{tikzpicture}[scale=1.5]
	\node at (-0.3,0) {$x$};
	\node at (2.3,0) {$y$};
	\node at (1.3,-1) {$z$};
	
	\node at (0.5,0.2) {$e_1$};
	
	\vertex (x) at (0,0) {};
	\vertex (w) at (1,0) {};
	
	\redvert (y) at (2,0) {};
	\redvert (z) at (1,-1) {};
	
	\path [line width=1.75pt]
	(x) edge (w)
	;  
	
	\path [gray,line width=1pt]
	(w) edge (y)
	(x) edge (z)
	(w) edge (z)
	;

	\end{tikzpicture}\]
	\caption{ $\tau = \{e_1\}$; $e_1$ is a leaf-edge in $G_\tau$ and $e_1$ contains $x$ as a leaf in $G_\tau$. Moreover, $e_1$ also contains neighbors of $y$ and $z$ as leaves.}\label{sigma5}
	\end{minipage}
\hfill
\begin{minipage}{0.45\textwidth}
	\[\begin{tikzpicture}[scale=1.5]
\node at (-0.3,0) {$x$};
\node at (3.3,0) {$y$};
\node at (1.3,-1) {$z$};

\node at (1.5,0.2) {$e_1$};
\node at (0.5,0.2) {$e_2$};

\vertex (x) at (0,0) {};
\vertex (w1) at (1,0) {};
\vertex (w2) at (2,0) {};

\redvert (y) at (3,0) {};
\redvert (z) at (1,-1) {};

\path [line width=1.75pt]
(x) edge (w1)
(w2) edge (w1)
;  

\path [gray,line width=1pt]
(w2) edge (y)
(x) edge (z)
(w1) edge (z)
(w2) edge (z)
;  
\end{tikzpicture}\]
\caption{ $\tau = \{e_1, e_2\}$; $e_1$ is a leaf-edge in $G_\tau$ and $e_1$ contains (some) neighbors of $y$ and $z$ as leaves.} \label{sigma6}
\end{minipage}	
	
\end{figure}

Similarly as before, we consider those simplices $\tau$ in $\mathcal{C}_1$ such that, $e_2$ is in $\tau$ and either a part of a cycle or a leaf-edge in $G_\tau$ such that $e_2$ doesn't contain $x$ as a leaf, and if $z \in V(G) \setminus V_\tau$ then $e_2$  doesn't contain any neighbor of $z$ as a leaf in $G_\tau$. We pair off $\tau$ with $\tau \setminus \{e_2\}.$ This extended matching is denoted by $\Phi_2.$

 As before, in the special case of $ \{e_2 \} \in \mathcal{C}_1$ and $\{e_2 \}, \emptyset \in \mathcal{A}_y \setminus \mathcal{A}_x$, we pair off $ \{e_2 \}$ with $\emptyset$.

\textbf{Step 3:} The set of all unpaired simplices after step $2$ is denoted by $\mathcal{C}_2.$ Note that $\mathcal{C}_2$ consists of simplices  $\eta$ in $\mathcal{C}_1$ if
one of the following four conditions holds:

 \begin{enumerate}
\item $e_2 \in \eta$ is a bridge but not a leaf-edge in $G_\eta$. 

\item $e_2 \in \eta$ is a leaf-edge in $G_\eta$ such that, 

    \begin{enumerate}
         \item  $e_2$ contains $x$ as a leaf  in $G_{\eta}$ \\ or
       \item  there exists $z \in V(G) \setminus V_\eta$ such that $e_2$  contains a neighbor of $z$ as a leaf in $G_\eta$. 

\end{enumerate}
\item $e_2 \notin \eta$ and $y \in e_2.$
\item $e_2 \notin \eta, y \notin e_2$ and $\eta \cup \{e_2\}$ contains a cycle containing $e_1.$

\end{enumerate}

Now we proceed as before, i.e., consider those simplices $\eta$ in $\mathcal{C}_2$ such that, $e_3$ is in $\eta$ and either  a part of a cycle or a leaf-edge in $G_\eta$ such that $e_3$ doesn't contain $x$ as a leaf, and if $z \in V(G) \setminus V_\eta$ then $e_3$  doesn't contain any neighbor of $z$ as a leaf in $G_\eta$. We pair off $\eta$ with $\eta \setminus \{e_3\}.$ This extended matching is denoted by $\Phi_3.$

Proceeding this way, we end up with a matching $\Phi$. To verify that $\Phi$ is acyclic, we first observe that at the $i$-th step, we pair a simplex $\Sigma$ with $\Sigma \setminus e_i$ only if $\operatorname{Shade}_x(\Sigma)=\operatorname{Shade}_x(\Sigma \setminus e_i)=E(G) $. More generally, we make the following observation.

\begin{obs} \label{pair}

We pair off   $\Sigma  \in \mathcal{A}_y \setminus \mathcal{A}_x$ with $\Sigma \setminus e$ only if $e$ is the minimal edge in $\Sigma$ such that $\operatorname{Shade}_x(\Sigma \setminus e)=E(G) $. 

\end{obs} 

We also make the following observations regarding a unpaired simplex $\Sigma$ in this matching $\Phi$.

\begin{obs}

If a simplex  $\Sigma \in \mathcal{A}_y \setminus \mathcal{A}_x $ is unpaired in $\Phi$, then the following properties hold.
\begin{enumerate}
	\item $G_\Sigma$ is a spanning tree of $G \setminus \{ y \}$. Therefore, it is a $(|V(G)|-3)$-simplex (i.e., it has $|V(G)|-2$ edges).
	\item For any $e_j \notin \Sigma$ with $y \notin e_j$, the unique cycle in $G_{\Sigma \cup e_j}$ has an edge $e_i$ such that $e_i<e_j$. 
\end{enumerate} 

\end{obs}

First, we consider the following alternating path with respect to the matching  $\Phi$ starting from $\sigma_1 \in \mathcal{A}_y \setminus \mathcal{A}_x$.

	\begin{tikzcd}
	& \Sigma_1 \arrow[rd]  & & \Sigma_2 \arrow[rd] && \Sigma_3 & \cdots & \Sigma_{k-1} \arrow[rd]
	\\
		\sigma_1 \arrow[ru] & & \sigma_1 \arrow[ru] & & \sigma_2 \arrow[ru] && \cdots & & \sigma_{k}   
	\end{tikzcd}
	
	Suppose $i$ is an integer  between $2$ and $k-1$ such that  $\sigma_i \in \mathcal{A}_x$. Then, we observe that $\sigma_j \in \mathcal{A}_x$ for all $j \in \{i, \cdots, k \}$ and  $\Sigma_{j} \in \mathcal{A}_x$ for all $j \in \{i, \cdots, k-1 \}$. It is already known that the matching $\Phi_0$ is acyclic in $\mathcal{A}_x$. Therefore, it is impossible to have closed $\Phi$-path containing simplices from both $\mathcal{A}_y \setminus \mathcal{A}_x$ and $\mathcal{A}_x$. So, it suffices to show that no closed $\Phi$-path exists containing simplices only from $\mathcal{A}_y \setminus \mathcal{A}_x$.\\

Now assume on the contrary that there is a directed cycle in the matching as follows.

	\begin{tikzcd}
		\Sigma_n \arrow[rd] & & \Sigma_1 \arrow[rd] & & \Sigma_2 \arrow[rd] && \cdots & & \Sigma_n   \\
		 & \sigma_1 \arrow[ru]  & & \sigma_2 \arrow[ru] && \sigma_3 & \cdots & \sigma_n \arrow[ru]
	\end{tikzcd}

Here, we assume $\Sigma_i, \sigma_i \in \mathcal{A}_y \setminus \mathcal{A}_x$ for each $i$. We note that for each $i$ there are edges $f(i) \notin \sigma_i$ and $g(i) \in \Sigma_i$, such that $\Sigma_i=\sigma_i \cup f(i)$ and $\sigma_{i+1}=\Sigma_i \setminus g(i)$.  
 It now follows that, $f(i)<g(i)$ (with respect to the ordering in $E(G)$) since otherwise it will violate Observation~\ref{pair} . We also assume without loss of generality that $f(1)= \min_{i} f(i) $. Since, $\{f(i): i=1,\cdots,n\}=\{g(i): i=1,\cdots,n\}$, there is $1 \le j \le n$ such that $f(1)=g(j)>f(j)$  which contradicts the fact that $f(1)$ is the minimum. Therefore, the matching $\Phi$ is an acyclic extension of the acyclic matching $\Phi_0$ for $\mathcal{A}_x$.

Now, let $U_n= \{ x_1,\cdots,x_n \}$. We observe that $\mathcal{A}_{U_n}=\mathcal{A}_{U_{n-1}} \cup  (\mathcal{A}_{x_n} \setminus \mathcal{A}_{U_{n-1}})$. Proceeding inductively, we construct an acyclic matching of $\mathcal{A}_U$ where $|U|>1$ that has critical simplices at dimension $|V(G)|-3$. 
 
By Theorem \ref{forth}, we conclude that $\mathcal{A}_U$ is homotopy equivalent to a wedge of spheres of dimension $|V(G)|-3$. Since, $\Delta^{G}_{U}$ is the Alexander dual of $\mathcal{A}_U$, it follows by Alexander duality that $\tilde{H}_k(\Delta^{G}_{U}) \cong \tilde{H}^{|E(G)|-k-3}(\mathcal{A}_U) \cong \tilde{H}_{|E(G)|-k-3}(\mathcal{A}_U) \cong 0 $ unless $|E(G)|-k-3=|V(G)|-3$. Therefore, $\tilde{H}_{k}(\Delta^{G}_{U};\mathbb{R}) \cong 0$ unless $k=|E(G)|-|V(G)|$.

\end{proof}

\end{document}

%% file: gradientfield.pdf_tex
\begingroup%
  \makeatletter%
  \providecommand\color[2][]{%
    \errmessage{(Inkscape) Color is used for the text in Inkscape, but the package 'color.sty' is not loaded}%
    \renewcommand\color[2][]{}%
  }%
  \providecommand\transparent[1]{%
    \errmessage{(Inkscape) Transparency is used (non-zero) for the text in Inkscape, but the package 'transparent.sty' is not loaded}%
    \renewcommand\transparent[1]{}%
  }%
  \providecommand\rotatebox[2]{#2}%
  \newcommand*\fsize{\dimexpr\f@size pt\relax}%
  \newcommand*\lineheight[1]{\fontsize{\fsize}{#1\fsize}\selectfont}%
  \ifx\svgwidth\undefined%
    \setlength{\unitlength}{249.35078961bp}%
    \ifx\svgscale\undefined%
      \relax%
    \else%
      \setlength{\unitlength}{\unitlength * \real{\svgscale}}%
    \fi%
  \else%
    \setlength{\unitlength}{\svgwidth}%
  \fi%
  \global\let\svgwidth\undefined%
  \global\let\svgscale\undefined%
  \makeatother%
  \begin{picture}(1,0.43610854)%
    \lineheight{1}%
    \setlength\tabcolsep{0pt}%
    \put(0,0){\includegraphics[width=\unitlength,page=1]{gradientfield.pdf}}%
    \put(-0.00272575,0.16318388){\color[rgb]{0,0,0}\makebox(0,0)[lt]{\lineheight{1.25}\smash{\begin{tabular}[t]{l}$v_1$\end{tabular}}}}%
    \put(0.25192556,0.12731002){\color[rgb]{0,0,0}\makebox(0,0)[lt]{\lineheight{1.25}\smash{\begin{tabular}[t]{l}$v_2$\end{tabular}}}}%
    \put(0.51276938,0.14709817){\color[rgb]{0,0,0}\makebox(0,0)[lt]{\lineheight{1.25}\smash{\begin{tabular}[t]{l}$e_1$\end{tabular}}}}%
    \put(0.36316409,0.06140012){\color[rgb]{0,0,0}\makebox(0,0)[lt]{\lineheight{1.25}\smash{\begin{tabular}[t]{l}$e_2$\end{tabular}}}}%
  \end{picture}%
\endgroup%

%% file: collapse.pdf_tex
\begingroup%
  \makeatletter%
  \providecommand\color[2][]{%
    \errmessage{(Inkscape) Color is used for the text in Inkscape, but the package 'color.sty' is not loaded}%
    \renewcommand\color[2][]{}%
  }%
  \providecommand\transparent[1]{%
    \errmessage{(Inkscape) Transparency is used (non-zero) for the text in Inkscape, but the package 'transparent.sty' is not loaded}%
    \renewcommand\transparent[1]{}%
  }%
  \providecommand\rotatebox[2]{#2}%
  \newcommand*\fsize{\dimexpr\f@size pt\relax}%
  \newcommand*\lineheight[1]{\fontsize{\fsize}{#1\fsize}\selectfont}%
  \ifx\svgwidth\undefined%
    \setlength{\unitlength}{456.39344672bp}%
    \ifx\svgscale\undefined%
      \relax%
    \else%
      \setlength{\unitlength}{\unitlength * \real{\svgscale}}%
    \fi%
  \else%
    \setlength{\unitlength}{\svgwidth}%
  \fi%
  \global\let\svgwidth\undefined%
  \global\let\svgscale\undefined%
  \makeatother%
  \begin{picture}(1,0.1493904)%
    \lineheight{1}%
    \setlength\tabcolsep{0pt}%
    \put(0,0){\includegraphics[width=\unitlength,page=1]{collapse.pdf}}%
  \end{picture}%
\endgroup%